\documentclass{amsart}
\usepackage[x11names, rgb]{xcolor}
\usepackage{amsmath, amssymb, amsthm, enumitem, tikz-cd, tikz, mathtools, tabularx,pdflscape,float, geometry}
\usepackage[utf8]{inputenc}

\usetikzlibrary{decorations,arrows,shapes}

\newgeometry{margin=1.5in}

\theoremstyle{definition}
\newtheorem{theorem}{Theorem}
\newtheorem{prop}[theorem]{Proposition}
\newtheorem{corollary}[theorem]{Corollary}
\newtheorem{lemma}[theorem]{Lemma}
\newtheorem{alg}{Algorithm}
\newtheorem{defn}[theorem]{Definition}
\newtheorem*{citethm}{Theorem}

\newtheorem*{nota}{Notation}

\theoremstyle{remark}

\DeclareRobustCommand{\bb}[1]{\mathbb{#1}}

\DeclareRobustCommand{\t}[1]{\text{#1}}

\DeclareRobustCommand{\set}[1]{\left\{#1 \right\}}
\DeclareRobustCommand{\Set}[2][]{\left\{#1 \, :\, #2 \right\}}

\DeclareRobustCommand{\a}{\alpha}

\DeclareRobustCommand{\card}[1]{\mathrm{card}\left(#1\right)}

\newcommand\restr[2]{{
  \left.\kern-\nulldelimiterspace 
  #1 
  \vphantom{\rule{0pt}{9pt}} 
  \right|_{#2} 
  }}

\usepackage{mathtools}

\DeclareRobustCommand{\rev}{\mathrm{Rev}\,}

\DeclareRobustCommand{\Inv}{\mathrm{Inv}\,}

\begin{document}

\title{Higher Bruhat Orders in Type B}\author{Seth Shelley-Abrahamson and Suhas Vijaykumar}\date{May 12, 2015}\begin{abstract} Motivated by the geometry of certain hyperplane arrangements, Manin and Schechtman  \cite{MS} defined for each integer $n \geq 1$ a hierarchy of finite partially ordered sets $B(n, k),$ indexed by positive integers $k$, called the higher Bruhat orders.  The poset $B(n, 1)$ is naturally identified with the weak left Bruhat order on the symmetric group $S_n$, each $B(n, k)$ has a unique maximal and a unique minimal element, and the poset $B(n, k + 1)$ can be constructed from the set of maximal chains in $B(n, k)$.  Elias \cite{E2} has demonstrated a striking connection between the posets $B(n, k)$ for $k = 2$ and the diagrammatics of Bott-Samelson bimodules in type A, providing significant motivation for the development of an analogous theory of higher Bruhat orders in other Cartan-Killing types, particularly for $k = 2$.  In this paper we present a partial generalization to type B, complete up to $k = 2$, prove a direct analogue of the main theorem of Manin and Schechtman, and relate our construction to the weak Bruhat order and reduced expression graph for Weyl group $B_n$.\end{abstract}\maketitle

\tableofcontents

\section{Higher Bruhat orders in type $A$ (cf. \cite{MS})}
In this section, we recall the construction of and main theorem for the original Manin-Schechtman higher Bruhat orders.

Let $I_n := \set{1, \ldots, n}$ be totally ordered in the usual way, and let $C(I_n,k)$ denote the set of $k$-element subsets of $I_n$. Then $C(I_n, k)$ is totally ordered by the lexicographic ordering, denoted $\rho_{min}$.  We refer to the reverse total ordering as the anti-lexicographic ordering, denoted $\rho_{max}.$

For each $K \in C(I_n, k+1)$, let $P(K) := \Set[S \in C(I_n,k)]{S \subset K}$ be the set of all $k$-element subsets of $K$. We refer to $P(K)$ as the \emph{packet of $K$} and to any subset of $C(I_n,k)$ of the form $P(K)$ for some $K \in C(I_n, k + 1)$ as a \emph{$k$-packet}. Call a total ordering of $C(I_n,k)$ \emph{admissible} if its restriction to each $k$-packet is either lexicographic or anti-lexicographic. Let $A(I_n,k)$ denote the set of admissible total orderings of $C(I_n,k)$. Clearly, if $\rho$ is an admissible ordering so is its reverse ordering $\rho^t$, and both the lexicographic ordering $\rho_{min}$ the anti-lexicographic ordering $\rho_{max} = \rho_{min}^t$ are admissible.  Note that any total ordering of $C(I_n, 1) = I_n$ is admissible, as the only total orderings on a 1-packet are the lexicographic and anti-lexicographic orderings, so the admissibility criterion is vacuous for $k = 1$ and $A(I_n, 1)$ is the set of total orderings of $I_n$.

For a set $S$, let $2^S$ denote the set of subsets of $S$, and for a total ordering $\rho$ on $S$ and a subset $T \subset S$ let $\rho|_T$ denote the restriction of $\rho$ to $T$.  Let the function $\Inv : A(I_n, k) \rightarrow 2^{C(I_n, k + 1)}$ be defined by $$\Inv(\rho) := \{K \in C(I_n, k + 1) : \rho|_{P(K)} = \rho_{max}|_{P(K)}\}.$$  For example, $\Inv(\rho_{min}) = \emptyset$ and $\Inv(\rho_{max}) = C(I_n, k + 1)$.  For $k = 1$ this gives the usual notion of the inversion set of a permutation of $I_n$.

Let also the function $N : A(I_n, k) \rightarrow 2^{C(I_n, k+ 1)}$ be defined by $$N(\rho) := \{K \in C(I_n, k + 1) : P(K) \text{ forms a chain in } \rho\}.$$  For $\rho \in A(I_n, k)$ and $K \in N(\rho)$, the ordering $p_K(\rho)$ obtained by reversing the ordering of the chain $P(K)$ in $\rho$ is also admissible, because any two $k$-packets have intersection of size at most 1.  This operation $p_K$, when defined, is called a \emph{packet flip}.  In particular, we see \[\Inv(\rho') = \begin{cases} \Inv(\rho) \setminus K & \t{ if }K \in \Inv(\rho) \\\Inv(\rho) \cup \set{K} & \t{ otherwise.} \end{cases}\]

We will now construct the set $B(I_n,k)$ as a quotient of $A(I_n, k)$ by a certain equivalence relation. Call $\rho, \rho' \in A(I_n, k)$ \emph{elementarily equivalent} if one can be obtained from the other by reversing the order of two neighboring elements that do not belong to a common $k$-packet. Let $\sim$ be the equivalence relation on $A(I_n, k)$ generated by these elementary equivalences, so that $\rho \sim \rho'$ if and only if $\rho$ and $\rho'$ can be connected by a sequence of elementary equivalences.  Let $$B(I_n, k) := A(I_n, k)/\sim$$ be the quotient of $A(I_n, k)$ by this equivalence relation.  For $\rho \in A(I_n, k)$ let $[\rho]$ denote its class in $B(I_n, k)$.  We set $r_{min} = [\rho_{min}]$ and $r_{max} = [\rho_{max}].$

It is clear that if $\rho \sim \rho'$ then $\Inv(\rho) = \Inv(\rho')$ and hence $\Inv$ descends to $B(I_n, k)$.  We also extend the definition of $N$, defining $N(r)$ for $r \in B(I_n, k)$ by $$N(r) = \cup_{\rho \in r} N(\rho).$$  For $K \in N(r)$ there exists $\rho \in r$ with $K \in N(\rho)$, so that $p_K(\rho)$ is defined.  It is clear then that $[p_K(\rho)]$ is independent of the choice of $\rho$, and we extend the definition of packet flips be defining $p_K(r) = [p_K(\rho)]$ for any $\rho \in r$ and $K \in N(\rho)$.

We now define the Manin-Schechtman higher Bruhat orders on the sets $B(I_n, k)$.  For $r, r' \in B(I_n, k)$, we write $r <_{MS} r'$ if there exist sequences $K_1,...,K_m \in C(I_n,k+1)$ and $r_0, ..., r_m \in B(I_n, k)$ such that $r = r_0$, $r' = r_m$, $K_i \in N(r_{i - 1}) \setminus \Inv(r_{i - 1})$, and $r_i = p_{K_i}(r_{i - 1})$ for $1 \leq i \leq m$. The following theorem was proven by Manin and Schechtman about the relation $<_{MS}$:

\citethm[cf. \cite{MS}] The following hold:$ $
\begin{itemize} 
\item $<_{MS}$ defines a partial order on $B(I_n,k)$.
\item Under $<_{MS}$, $B(I_n,k)$ is a ranked poset with a unique minimal element, $r_{min}$, and a unique maximal element, $r_{max}$. The rank is given by $r \mapsto |\Inv(r)|$.
\item The map $r_{min} < p_{K_1}(r_{min}) < \cdots < p_{K_m}\cdots p_{K_1}(r_{min}) \mapsto \rho: K_1 \prec \cdots \prec K_m$ defines a bijection from the set of maximal chains in $B(I_n,k)$ to the set $A(I_n,k+1)$.
\item The map $\Inv: B(I_n,k) \to 2^{C(I_n,k+1)}$ is injective.
\end{itemize}
\endcitethm
This is the central result that we wish to generalize to type $B$.  First, we make explicit the connection of this construction for $k = 1$ and $k = 2$ with the type $A$ Weyl groups $S_n$.

\section{Connection with Type A Weyl Groups}
As we saw before, for $k = 1$ admissibility is a vacuous condition on orders of $I$, so $A(I_n, 1)$ is the set of total orderings of $I_n$.  Furthermore any two distinct elements $i \neq j \in I_n$ belong to the common packet $\{i,j\}$ so there are no elementary equivalences between orderings, and $B(I_n, 1) = A(I_n, 1)$.  We may then identify $B(I_n, 1)$ with the symmetric group $S_n$, where the total ordering $a_1 < \cdots < a_n$ corresponds to the permutation $a_i \mapsto i$.

Let $$\Phi := \{\pm(e_i - e_j): 1 \leq i < j \leq n\} \subset \mathbb{R}^n$$ be the root system attached to $S_n$, and choose the system of positive roots $$\Phi^+ := \{e_i - e_j : 1 \leq i < j \leq n\}$$ with associated simple roots $\Pi = \{e_i - e_{i + 1} : 1 \leq i < n\}$ and simple reflections $s_i = (i, i + 1)$ for $1 \leq i < n$.  There is then an obvious bijection between $C(I_n, 2)$ and $\Phi^+$ given by $$\{i, j\} \mapsto e_i - e_j$$ for $i < j$.  We have a function $\Inv : S_n \rightarrow 2^{\Phi^+}$ given as usual by $$\Inv(w) := \{\alpha \in \Phi^+ : w(\alpha) \notin \Phi^+\}$$ so the length function $l : S_n \rightarrow \mathbb{Z}^{\geq 0}$ is $l(w) = |\Inv(w)|$.  This is compatible with $\Inv : B(I_n, 1) \rightarrow 2^{C(I_n, 2)}$ in the sense that the following diagram commutes: 

\begin{center}
\begin{tikzpicture}
  \matrix (m) [matrix of math nodes,row sep=3em,column sep=4em,minimum width=2em]
  {
     B(I_n, 1) & S_n \\
     2^{C(I_n, 2)} & 2^{\Phi^+} \\};
  \path[-stealth]
    (m-1-1) edge node [left] {$\Inv$} (m-2-1)
            edge [double] node [above] {$\sim$} (m-1-2)
    (m-2-1.east|-m-2-2) edge node [below] {}
            node [above] {$\sim$} (m-2-2)
    (m-1-2) edge node [right] {$\Inv$} (m-2-2);
\end{tikzpicture}
\end{center}

For an ordering $\rho = (a_1 < \cdots < a_n)$ of $I_n$ corresponding to permutation $w \in S_n$, the action of the packet flip $p_{\{a_i, a_{i + 1}\}}$ for $i < n$ on $\rho$ corresponds to left multiplication by the adjacent transposition $s_i := (i, i + 1) \in S_n$.  One then sees immediately that the ordering $<_{MS}$ on $B(I_n, 1)$ is identified with the weak left Bruhat order on $S_n$, defined by the covering relations $w < w'$ for $l(w') = l(w) + 1$ and $w' = s_iw$ for some $i$.

The unique minimal element $r_{min}$ of $B(I_n, 1)$ is identified with $1 \in S_n$ and the unique maximal element $r_{max}$ is identified with the longest element $w_0 \in S_n$ given by $w_0(i) = n + 1 - i$.  The identification of $A(I_n, 2)$ with the set of maximal chains in $B(I_n, 1)$ therefore identifies $A(I_n, 2)$ with the set of reduced expressions for $w_0$.  Elementary equivalence in $A(I_n, 2)$ then corresponds to exchanging the positions of two adjacent commuting simple reflections $s_i$ and $s_j$ for $|i - j| > 1$, and packet flip operations correspond to $m = 3$ braid relations $s_is_{i + 1}s_i = s_{i + 1}s_is_{i + 1}$.

In the next section, we introduce a partial generalization of the Manin-Schechtman construction to type B.  We provide analogues of the sets $C(I_n, k)$ and packets for all $k$ and $n$, and for $k = 1, 2$ and all $n$ we give a complete analogue, providing partially ordered sets analogous to the $B(I_n, k)$ satisfying the direct analogues of the main theorem of Manin and Schechtman on $<_{MS}$ and the root system combinatorics explained above in type $A$.  In type $B$, $m = 4$ ($stst = tsts$) braid relations become relevant in addition to the $m = 3$ ($sts = tst$) braid relations seen in type $A$, leading to the introduction of two types of packets.

\section{Construction in Type B}

Let $E$ be a finite subset of $\bb Z$ stable under negation and not containing $0$.  Let $E_+$ denote the subset of positive elements in $E$.  We give $E$ the total ordering inherited from $\bb Z$.

\defn Let $C_B(E, 1) := E$. For $k > 1$, let

\begin{align*} & \widetilde{C_B^1(E,k)} := \Set[S \subset E]{\card{S}=k, \t{ and the elements of $S$ have distinct absolute value} },  \\ & C_B^2(E,k) := \Set[T \cup \set{\star}]{T \subset E_+,\, \card{T} = k-1} \end{align*}
where $\star$ is a bookkeeping symbol.  $\bb Z/2\bb Z$ acts on $\widetilde{C_B^1(E, k)}$ by negation, and we let $$C_B^1(E,k) = \widetilde{C_B^1(E, k)}/(\bb Z/2\bb Z).$$  Finally, we set $$C_B(E, k) = C_B^1(E, k) \cup C_B^2(E, k).$$
\enddefn

Let $J_n = \{-n, ..., -1, 1, ..., n\}$ for $n \geq 1$.  The sets $C_B(J_n, k)$ are to play the roles of the sets $C(I_n, k)$ seen previously.  We use the notation $C_B$ to indicate that these are constructions related to type $B$.  As in the type $A$ case, we want to associate to an element $K \in C_B(J_n, k+1)$ a subset, $P_B(K) \subset C_B(J_n, k)$, called its \emph{packet}.  As before, a subset of $C_B(J_n, k)$ of the form $P_B(K)$ will be called a $k$-\emph{packet}.  $P_B(K)$ will be constructed differently depending on whether $K \in C_B^1(J_n, k)$ or $K \in C_B^2(J_n, k)$.

For brevity, we will define $\sigma$ to be the negation map $x \mapsto -x$.

\defn For $K \in C_B^1(J_n,k+1)$, let $R \in \widetilde{C_B^1(J_n, k + 1)}$ be a representative. Define $\widetilde P_B(K)$ to be the set of the $\sigma$-orbits of the $k$-element subsets of $R$. Clearly, $\widetilde P_B(K)$ is independent of the choice of representative $R$.  Note that for $k = 1$, the elements of $C_B(J,k)$ are not themselves $\sigma$-orbits. For this reason, we make the correction \[P_B(K) := \begin{cases} \cup_{T \in \widetilde P_B(K)} T  & \t{if $K \in C_B^1(J_n, 2)$} \\ \widetilde P_B(K) & \t{if $K \in C_B^1(J_n, k + 1)$ for $k > 1$} \end{cases}\] For example, if $K = [\{-3, 2\}]$, then $$P_B(K) = \{-3, -2, 2, 3\} \subset C_B(J_3, 1),$$ and if $K = [\{-3, 2, 1\}]$ then $$P_B(K) = \{[\{-3, 2\}], [\{-3, 1\}], [\{2, 1\}]\} \subset C_B(J_3, 2).$$

For $K \in C_B^2(J_n,k+1)$, let $K' = K \setminus \set{\star}$, and define $P_B(K) := C_B(K' \cup \sigma(K'),k)$.
\enddefn

\nota
As $C_B^1(J_n,k)$ is a set of equivalence classes, we encounter the problem of choosing a good way to represent its elements. We will denote the class of $\{a_1, ..., a_k\}$ by $[a_1, ..., a_k]$.  Wherever possible, we will choose the representative for which the element with the greatest magnitude is negative. For some $T \in C_B^1(J_n,k)$, if $\set{a_1, \ldots, a_k}$ is such a representative, we will denote $T$ by the bracketed list $[a_1, \ldots, a_k]$. {\bf Such a representative will be referred to as an \emph{preferred} representative, and we will indicate where this choice of representative is assumed.}

For consistency, an element $s \in C_B^2(J_n,k)$ will be denoted as a bracketed list $[b_1,\ldots b_{k-1}, \star]$, as well, with either all $b_i$ positive or all $b_i$ negative.

 We list elements with negative elements first, in increasing order, followed by positive elements, in decreasing order, e.g. $[-5,-2,3,1]$.
\endnota

\defn We now introduce certain \emph{standard orders} on the sets $C_B(J_n, k)$ for $k = 1, 2, 3$.

The standard order of $C_B(J_n, 1) = J_n$ is the usual ordering of $J_n$ inherited from $\bb Z$.

The standard order of the set $C_B(J_n,2)$ is defined with respect to preferred representatives. It is convenient here to represent elements $[a_1,\,\star]$ by $[-a_1,\,a_1]$.  Viewing a total ordering as a list of elements read left to right, with the smallest elements occurring first, i.e. to the left, we define the standard order on $C_B(J_n, 2)$ as follows:

\begin{itemize}
\item Elements represented by two negative indices occur first, in lexicographic order. 
\item Elements with a single negative index occur afterwards. If the elements are listed in increasing order, $a_1,\,a_2$, then the ordering is lexicographic in the following sense: \[ [a_1,\, a_2] < [b_1,\,b_2] \t{ if } a_1 < b_1, \t{ or if } a_1 = b_1 \t{ and } a_2 > b_2 \]
\end{itemize}

Similarly, we have the following standard order for the set $C_B(J_n,3)$. Similar to the above, it is convenient here to represent elements of the form $[a_1,\,a_2,\,\star]$ by $[-a_1,\,-a_2,\,a_2]$ with $a_1 > a_2 > 0$.

\begin{itemize}
\item Elements represented by three negative indices occur first, in lexicographic order.
\item Elements represented by two negative indices occur second. For two elements represented by negative indices, define the order by \[[a_1,\,a_2,\,a_3] < [b_1,\,b_2,\,b_3] \t{ if } [a_1,\,a_2] < [b_1,\,b_2],\t{ or if } [a_1,\,a_2] = [b_1,\,b_2] \t{ and }a_3 > b_3\,\] \hfill where $a_3,\,b_3 > 0$.

\smallskip

\item Elements represented by a single negative index occur third. For two elements of this form, we set \[[a_1,\,a_2,\,a_3] < [b_1,\,b_2,\,b_3] \t{ if } a_1 > b_1,\t{ or if } a_1 = b_1 \t{ and }[-a_2,\,-a_3] < [-b_2\,,-b_3] \t.\]
\end{itemize}

The standard ordering of a given set is denoted $\rho_{min}$ whenever this notation is unambiguous. For these sets, the \emph{reverse standard ordering}, $\rho_{max},$ is obtained by reversing the standard ordering.  \enddefn

\defn For a $k$-packet $P$ with $k = 2$ or $k = 3$, the \emph{standard ordering} of $P$ is given by the restriction of the standard ordering $\rho_{min}$ to $P$. For a $1$-packet, standard packet orderings are given by the following Hasse diagrams, where we assume $[j, i]$ is a preferred representative as above:
\begin{center}
\begin{figure}[!htb] \hfill
\minipage{0.2\textwidth}
\endminipage\hfill
\minipage{0.2\textwidth}
  \center{$[j,\,i]$} \\

  \medskip

  \begin{tikzpicture}
  \node (i) at (-1,0) {$i$};
  \node (mi) at (0,-1) {$-i$};
  \node (j) at (-1,-1) {$j$};
  \node (mj) at (0,0) {$-j$};
  \draw[->] (j) -- (i);
  \draw[->] (mi) -- (mj);
\end{tikzpicture}
\endminipage\hfill
\minipage{0.2\textwidth}
  \center{$[i,\,-j]$} \\

  \medskip

      \begin{tikzpicture}
  \node (k) at (-1,0) {$k$};
  \node (mk) at (-1,-1) {$-k$};
  \draw[->] (mk) -- (k);

\end{tikzpicture}
\endminipage\hfill
\minipage{0.2\textwidth}
\endminipage\hfill
\end{figure}
\end{center}

The \emph{reverse ordering} of a 1-, 2-, or 3-packet $T$ is given by reversing the direction of each inequality. By abuse of notation, $T$ will often denote both the set and the ordering relation, and $\rev T$ denotes both the set and the reverse ordering.
\enddefn

\defn A \emph{comparable component} of a poset is defined to be a connected component of the poset's Hasse diagram. \enddefn

\defn  For $k \leq 3$, we now make a collection of definitions analogous to the type A case:

\begin{itemize}
\item  A total ordering $\rho$ of $C_B(J_n,k)$ is \emph{admissible} if, for each $k$-packet $P$, $\rho$ extends either $P$ or $\rev P$.
\item Let $A_B(J_n,k)$ be the set of all admissible orderings of $C_B(J_n,k)$. In particular $\rho_{min}, \rho_{max} \in A_B(J_n, k).$
\item For an admissible ordering $\rho$, let $\Inv(\rho)$ be the set of elements $K \in C_B(J_n,k+1)$ such that $\rho$ extends $\rev P_B(K)$.
\item For $\rho \in A_B(J_n,k)$, let $N(\rho)$ be the set of $K \in C_B(J_n,k+1)$ such that for each comparable component $C \subset P_B(K)$, $C$ forms a chain in $\rho$.
\item Two elements of $C_B(J_n,k)$ \emph{commute} if they are incomparable in each $k$-packet to which they both belong.
\item Two orderings $\rho,\,\rho' \in A_B(J_n,k)$ are \emph{elementarily equivalent} if $\rho'$ can be obtained from $\rho$ by exchanging the order of two adjacent, commuting elements.
\item Let $\sim$ be the equivalence relation on $A_B(J_n,k)$ generated by elementary equivalence.
\item Let $B_B(J_n,k) = A_B(J_n, k)/\sim$. Let $[\rho]$ denote the equivalence class of $\rho \in A_B(J_n, k)$.
\item Let $N([\rho]) = \cup_{\rho' \in [\rho]} N(\rho')$.
\end{itemize}
\enddefn

As before, it is clear that $\rho \sim \rho'$ implies $\Inv(\rho) = \Inv(\rho')$, so $\Inv$ descends to $B_B(J_n,k)$. Moreover, we have the following result, as in type A:

\begin{prop}
The function $\Inv$ is injective on the set $B_B(J_n,k)$.
\end{prop}
\begin{proof}
Consider two orderings $\rho$ and $\rho'$, such that $\Inv(\rho) = \Inv(\rho') = S$. Note that the transitive closure of the union over the ordering relations $\rev P$ for $P \in S$ and $Q$ for $Q \in C(J_n,k+1) \setminus S$ defines a poset structure on $C_B(J_n,k)$, and both $\rho$ and $\rho'$ must be linear extensions of this poset. Furthermore, if two elements are incomparable in this poset, then they must be incomparable in every packet to which they both belong. But any two linear extensions of a finite poset differ by a sequence of transpositions of adjacent elements incomparable in the poset, so $\rho \sim \rho'$ as needed.\end{proof}

\begin{defn}
Given $\rho$ in $A(J_n,k)$ and $K \in N(\rho)$, we can construct a new admissible order $p_K(\rho),$ the \emph{packet flip} of $\rho$ by $K$, by reversing the order of each comparable component of $P_B(K)$ in $\rho$. Clearly \[\Inv(\rho') = \begin{cases} \Inv(\rho) \setminus K & \t{ if }K \in \Inv(\rho) \\\Inv(\rho) \cup \set{K} & \t{ otherwise.} \end{cases}\]  Like in type A, we may extend this operation to $B_B(J_n, k)$ by acting on representatives.  Specifically, for $r \in B_B(J_n, k)$ and $K \in N(r)$, there exists $\rho \in r$ with $K \in N(\rho)$, and we set $p_K(r) = [p_K(\rho)].$

 For any $[\rho],\,[\rho'] \in B_B(J_n,k)$, we write $[\rho] < [\rho']$ if there exists a finite sequence $\{\rho_i\}_1^m \subset A_B(J_n,k)$ such that $\rho_1 = \rho$, $\rho_m = \rho'$, and for each pair $(\rho_i,\,\rho_{i+1})$ there exists some $K_i \in N([\rho_i]) \setminus \Inv([\rho_i])$ such that for some $\rho'_i \in [\rho_i]$, $\rho_{i+1} = p_{K_i}(\rho'_i)$. This relation defines a partial ordering on the set $B_B(J_n,k)$.
\end{defn}

\theorem \label{t1} For the cases $k = 1,\,2$, $B_B(J_n,k)$ has a unique maximal (respectively, minimal) element, given by $[\rho_{max}]$ (resp. $[\rho_{min}]$).\endtheorem

Our proof will make use of the following lemmas, which will be proved after the general argument is given.

\nota For some admissible ordering $\rho \in A_B(J_n,k)$, and some $S \subset C_B(J_n,k)$, we write $\overline S(\rho)$ to denote the minimal chain containing $S$ in $\rho$. To simplify the notation, we write $S_K$ to denote the set $P_B(K)$.

\defn For $\rho \in A_B(J_n,2)$ we say that $x$ \emph{blocks} $S$ in $\rho$ when $x \notin S$ but $x \in \overline S(\rho')$ for all $\rho' \in [\rho]$. \enddefn

\lemma \label{l1}  Let a set $S \subset C_B(J_n,2)$ be given, and let $\rho \in A_B(J_n, 2)$. If $x$ does not block $S$ in $\rho$ and $x \notin S$, then there exists some $\rho' \in [\rho]$ such that $\overline S(\rho') \subset \overline S(\rho)$, and $x \not\in \overline S(\rho')$.\endlemma

\lemma \label{l2} Let $\rho \in A_B(J_n, 2)$.  For $K \in C_B(J_n, 3)$, $K \not\in N([\rho])$ if and only if there exists some $x$ which blocks $S_K$ in $\rho$.\endlemma

\lemma \label{l3} Let $[\rho] \in B_B(J_n, 2)$, and suppose $K \not\in N([\rho]) \cup \Inv([\rho])$.  Then at least one of the following seven cases holds for all $\rho' \in [\rho]$:

\begin{itemize}
\item If $K \in C_B^1(J_n,3)$, let $[i,\,j,\,k]$ be a preferred representative. Then, we have either:
\smallskip
    \begin{enumerate}[leftmargin=2cm]
    \item $[i,\,j] < [i,\,x] < [i,\,k]$, 
    \item $[i,\,k] < [k,\,x] < [j,\,k]$, or
    \item $[i,\,j] < [j,\,x] < [j,\,k]$,
    \end{enumerate}
    \hfill for some $x \in J_n \setminus \set{i,\,j,\,k}$.

    \smallskip

\item If $K \in C_B^2(J_n,3)$, fix $K = [i,\,j,\,\star]$. Then, we have either:
\smallskip
    \begin{enumerate}[leftmargin=2cm]
    \item $[i,\,j] < [i,\,x] < [i,\,\star]$,
    \item $[i,\,\star] < [i,\,x] < [i,\,-j]$,
    \item $[i,\,-j] < [j,\,x] < [j,\,\star]$, or
    \item $[i,\,j] < [j,\,x] < [i,\,-j]$,
    \end{enumerate}
    \hfill for some $x \in J_n \setminus \set{i,\,j,\,-i,\,-j}$.
\end{itemize} \endlemma

\lemma \label{l4} In the setting of the previous lemma, there exists some $K' \in C_B(J_n, 3) \setminus \Inv([\rho])$ such that either $\overline {S_{K'}}(\rho') \subsetneq \overline S_K(\rho')$ for all $\rho' \in [\rho]$, or the minimal element of $S_{K'}$ is greater than the minimal element in $S_{K}$ in every ordering $\rho' \in [\rho]$.\endlemma

\begin{proof}[Proof of Theorem 9]

Clearly any class of orderings $[\rho]$ satisfying $\Inv([\rho]) = C_B(J_n,k+1)$ must be maximal, so in particular $[\rho_{max}]$ is a maximal element. By injectivity of $\Inv$ on $B_B(J_n,k)$, $[\rho_{max}]$ is the unique class of orderings for which $\Inv([\rho]) = C_B(J_n,k+1)$. Therefore, we need only prove that any $[\rho] \in B_B(J_n, k)$ for which $\Inv([\rho]) \subsetneq C_B(J_n,k+1)$ is not maximal.  Let $[\rho]$ be such an ordering, so we have some $K \in C_B(J_n,k+1) \setminus \Inv([\rho])$.  We need to find $K' \in N([\rho])\setminus \Inv([\rho]).$

For $k = 2$ this follows immediately from the preceding lemma and induction.

For $k = 1$, the statement follows immediately from the identification in the next section (independently of the intervening material) of the poset $B_B(J_n, 1)$ and the Weyl group $B_n$ with the weak left Bruhat order.
\end{proof}

\begin{proof}[Proof of Lemma \ref{l1}]
Let $\rho, S$, and $x$ be as in the statement of the lemma. Then there exists $\hat \rho \in [\rho]$ such that $x \not\in \overline S(\hat\rho)$. By taking reverse orderings if necessary, we may assume $x < \min S$ in $\hat\rho$. Let $T$ be the subset of $\overline S(\rho)$ which is less than or equal to $x$ in the order $\rho$. Clearly $T$ forms a chain in $\rho$.

As $\hat{\rho} \sim \rho$, there exists a sequence $t_1 \ldots t_r$ of pairs of commuting elements of $C_B(J_n, 2)$ such that $\hat{\rho}$ can be obtained from $\rho$ by exchanging the order of the pair $t_1$, then $t_2$, etc., where at each step the pair $t_i$ to be reversed is an adjacent pair.  Let $t_{i_1}, ..., t_{i_s}$, with $1 \leq i_1 < i_2 < \cdots < i_s \leq r$ be the subsequence of pairs of elements in $T$.  Then the ordering $\rho'$ obtained from $\rho$ by first reversing the ordering of $t_{i_1}$, then $t_{i_2}$, etc., is such that $\bar{S}(\rho') \subset \bar{S}(\rho)$ and $x \notin \bar{S}(\rho')$.
\end{proof}

\begin{proof}[Proof of Lemma \ref{l2}]
Fix $S = P_B(K)$. If some $x$ blocks $S$ in $\rho$, then certainly $K \not\in N([\rho])$. Suppose that no element blocks $K$ in $\rho$. If $\overline S(\rho) = S$, then $K \in N([\rho])$. If not, there exists $y \in \overline S(\rho) \setminus S$. By Lemma 11, there exists $\rho' \sim \rho$ such that $y \not\in \overline S(\rho')$, and $\overline S(\rho') \subset \overline S(\rho)$, and the lemma follows by induction on $|\bar{S}(\rho)|$.
\end{proof}

\begin{proof}[Proof of Lemma \ref{l3}]
Let $\rho \in A_B(J_n,2)$ and $K \in C_B(J_n,3)\setminus (N([\rho]) \cup \Inv(\rho))$. Write $K = [i,\,j,\,k]$, where $i <j< 0$, by the conventions introduced earlier. Once again, let $S = P_B(K)$, and $\overline S(\rho)$ is the minimal chain containing $S$ in $\rho$.

First, we show that there exists some element which blocks $S$ in $\rho$, which does not commute with every element of $S$. Suppose to the contrary that every element which blocks $S$ commutes with every element of $S$. Then, by applying Lemma 11, we can produce some $\rho' \in [\rho]$ for which the only elements in $\overline S(\rho') \setminus S$ are those which block $\rho$ in $S$. But then as these elements commute with every element of $S$ there exists an equivalent ordering $\rho^*$ for which $\overline S(\rho^*) = S$, contradicting $K \notin N([\rho])$.

Suppose $K \in C_B^1(J_n,3)$. Then, we can conclude from the above that there exists some $b$ which blocks $S$ in $\rho$, and $b$ has the form $[i,\,x],\,[j,\,x]$, or $[k,\,x]$ for some $x \in J_n \setminus\{i, j, k\}$. Either $b$ falls into one of the stated cases, or one of the following:
\begin{enumerate}
\item $[i,\,j] < [k,\,x]  < [i,\,k]$
\item $[i,\,k] < [i,\,x]  < [j,\,k]$
\end{enumerate}
 In case (1), consider the set $D = \Set[d]{[i,\,j]<d<[k,\,x]}$. If every element of $D$ commutes with $[i,\,j]$, then, as $b = [k,x]$ also commutes with $[i, j]$, there exists an equivalent ordering $\rho' \sim \rho$ for which $\overline{S_K}(\rho')$ does not contain $b$, a contradiction. We may thus conclude that there is an element $b' \in D$ of the form $[i,\,x]$ or $[j,\,x]$ such that $[i,\,j] < b' < [i,\,k]<[j,\,k]$.

 In case (2), consider the set $D = \Set[d]{[i,\,k]<d<[j,\,k]}$. If every element of $D$ commutes with $[j,\,k]$, then there exists an equivalent ordering $\rho' \sim \rho$ for which $\overline{S_K}(\rho')$ does not contain $[i,\,x]$, a contradiction. We may thus conclude that there is an element $b' \in D$ of the form $[k,\,x]$ or $[j,\,x]$ such that $[i,\,j]<[i,\,k] < b' < [j,\,k]$.

 If $K \in C_B^2(J_n,3)$, then as above there exists some $b \in \overline S(\rho)$ of the form $[i,\,x]$ or $[j,\,x]$. Either $b$ belongs to one of the stated cases, or we have $[i,\,-j] < [i,\,x] < [j,\,\star]$. In this case, consider the set $D = \Set[d]{[i,\,-j]<d<[j,\,\star]}$. If every element of $D$ commutes with $[j,\,\star]$, then there exists an equivalent ordering $\rho' \sim \rho$ for which $\overline{S_K}(\rho')$ does not contain $[i,\,x]$, a contradiction. We may thus conclude that there is an element $b'$ of the form $[j,\,x]$ such that $[i,\,-j] < b' < [j,\,\star]$.
\end{proof}

\begin{proof}[Proof of Lemma \ref{l4}]
Lemma \ref{l4} is proved by case work. For the complete case analysis, refer to the Appendix [Sec. \ref{l4p}].
\end{proof}

\theorem \label{t2} For $k = 1,\,2$, there is a bijection \[\{\t{maximal chains in }B_B(J_n,k)\} \xrightarrow{\scalebox{1.2}{$\sim$}} A_B(J_n,k+1)\t,\] defined by \[ [\rho_{min}] = [\rho_1] \leq [\rho_2] \leq \cdots  [\rho_i] \cdots \leq [\rho_{max}] \mapsto K_1 < \cdots < K_m\t,\] where $[\rho_i] \in p_{K_{i-1}}(\dots p_{K_1}([\rho_{min}]))$, $m = |C(J_n,k + 1)|$ and $K_i \in N([\rho_i])\setminus \Inv([\rho_i])$ for all $i$. \endtheorem

\begin{proof} As $[\rho_{min}]$ is the unique minimal element of $B_B(J_n, k)$ and $[\rho_{max}]$ is the unique maximal element, the assignment \[ [\rho_{min}] = [\rho_1] \leq [\rho_2] \leq \cdots  [\rho_i] \cdots \leq [\rho_{max}] \mapsto K_1 < \cdots < K_m \t,\] maps the set of maximal chains of $B_B(J_n, k)$ into total orderings of $C_B(J_n, 3)$, and this map is clearly injective.  We need only show that its image is precisely $A_B(J_n, k + 1)$.  This reduces to checking a few cases, which is treated in the Appendix [Sec. \ref{t2p}].\end{proof}

\section{Connection with Type B Weyl Groups}

As before, let $J_n$ be the set $\set{-n \ldots n} \setminus \set{0}$. Recall that the Weyl group $B_n$ acts faithfully on $J_n$ as the set of permutations $\pi : J_n \rightarrow J_n$ such that $\pi(-i) = -\pi(i)$ for all $i \in J_n$. In this way, we have a natural inclusion $B_n \hookrightarrow S_{2n}$.  Likewise, the set $A_B(J_n, 1) = B_B(J_n, 1)$ of all admissible orderings of $J_n$ includes into the set $A(J_n, 1) = B(J_n, 1)$ of \emph{all} total orderings of $J$.

\begin{defn} For a total ordering $\rho$ of $J_n$, say $$j_{-n} < j_{-n + 1} < \cdots < j_{n - 1} < j_n,$$ let $\pi_\rho \in S_{2n}$ be the permutation of $J_n$ given by $j_i \mapsto i$.  Let $\phi : B(J_n, 1) \rightarrow S_{2n}$ denote the map $\rho \mapsto \pi_\rho$.
\end{defn}

\begin{prop} \label{prp:ims}
The image of the composition $B_B(J_n,1) \hookrightarrow B(J,1) \xrightarrow{\phi} S_{2n}$ is $B_n$
\end{prop}
\begin{proof} From the definitions, it is clear that the set of total orderings of $J_n$ sent into $B_n$ under $\phi$ are those reversed under negation.  Certainly any such total ordering belongs to $B_B(J_n, 1)$, and we need only show the converse.  So, let $\rho \in A_B(J_n,1)$ be given, and let $x$ be the maximal element of $J_n$ with respect to $\rho$. Then, for every other element $y \in J_n \setminus \set x$, considering the packet order on $P_B([x, y])$, we have $y < x$ and hence $-x < -y$. So $-x$ is the minimal element, and the claim follows by induction on $n$.
\end{proof}

Recall the reflection representation of $B_n$ in $\bb R^n$ given by \[e_i \mapsto \mathrm{sign}\,(\pi(i)) \cdot e_{|\pi(i)|}\] for $\pi \in B_n$, where $e_1, ..., e_n$ is the standard basis of $\bb R^n$.  We have the associated root system $$\Phi := \{\pm e_i : 1 \leq i \leq n\} \cup \{\pm e_i \pm e_j : 1 \leq i < j \leq n\}\subset \bb R^n.$$ We choose the set of positive roots $$\Phi^+ := \{e_i : 1 \leq i \leq n\} \cup \{e_i \pm e_j : i > j\}$$ with associated set of simple roots $$\Pi := \{e_1\} \cup \{e_i - e_{i - 1} : 1 < i \leq n\}.$$  Under the realization of $B_n$ with the subgroup of the permutations of $J_n$ discussed above, the simple reflection $s_{e_1}$ is given by the permutation $(-1, 1)$, and the simple reflection $s_{e_i - e_{i - 1}}$ for $1 < i \leq n$ is given by $(-i, -i + 1)(i, i - 1)$.  We define the function $$\Inv : B_n \rightarrow 2^{\Phi^+}$$ by $$\Inv(w) = \{\alpha \in \Phi^+ : w(\alpha) \notin \Phi^+\}$$ so that the length function $l : B_n \rightarrow \mathbb{Z}$ defined by the simple reflections $s_\alpha$ for $\alpha \in \Pi$ is given by $l(w) = |\Inv(w)|$.

We now give a specific bijection between $\Phi^+$ and $C_B(J_n, 2)$ such that the definitions of $\Inv$ on $\Phi^+$ and on $C_B(J_n, 2)$ become compatible with the identification of $B_B(J_n, 1)$ and $B_n$:

\defn Let $K \mapsto \alpha_K$ denote the bijection $C_B(J_n,2) \to \Phi^+$ given by \[ \alpha_K = \begin{cases} e_i - e_j & \t{if $K = [i,\,j]$ for $i > j > 0,$} \\ e_i + e_j & \t{if $K = [i,\,-j]$ for $i > j > 0,$} \\ e_k & \t{if $K = [k,\,\star]$ for $k > 0.$} \end{cases}\] \enddefn

\begin{lemma} \label{prp:roots}
Let $\rho \in B_B(J_n, 1)$, and let $\pi_{\rho} = \phi(\rho) \in B_n$ be the corresponding element of $B_n$.  Then for each $K \in C_B(J_n, 2)$, we have \[K \in \Inv(\rho) \iff \pi_\rho(\a_K) \not\in \Phi^+ \t.\]

\noindent In other words, the following diagram commutes.

\begin{center}
\begin{tikzpicture}
  \matrix (m) [matrix of math nodes,row sep=3em,column sep=4em,minimum width=2em]
  {
     B_B(J_n, 1) & B_n \\
     2^{C_B(J_n, 2)} & 2^{\Phi^+} \\};
  \path[-stealth]
    (m-1-1) edge node [left] {$\Inv$} (m-2-1)
            edge [double] node [below] {$\sim$} node [above] {$\phi$} (m-1-2)
    (m-2-1.east|-m-2-2) edge node [below] {$\sim$}
            node [above] {$\phi$} (m-2-2)
    (m-1-2) edge node [right] {$\Inv$} (m-2-2);
\end{tikzpicture}
\end{center}
\end{lemma}

\begin{proof}  Let $\rho$ and $\pi_{\rho}$ be as in the statement of the lemma, and let $i,\,j \in \{1, ..., n\}$ with $i > j > 0$. Let $k = \pi_{\rho}(i)$ and $l = \pi_\rho(j)$. Suppose that $[-j,\,-i] \not\in \Inv(\rho)$, so that $k > l$. Then the image of the positive root $\a_P = e_i - e_j$ under $\pi_\rho$ is positive. In particular, either: \begin{enumerate} 
\item $k > l > 0$, and $\pi_\rho(\a_P) = e_k-e_l$.
\item $k > 0 > l$, and $\pi_\rho(\a_P) = e_k + e_{-l}$, or 
\item  $0 > k >  l$, and $\pi_\rho(\a_P) = e_{-l} - e_{-k}$.
\end{enumerate} 
Conversely, if $[-i,\,-j] \in \Inv(\rho)$, so that $l > k$, we have one of the following:
\begin{enumerate} 
\item $l > k > 0$, and $\pi_\rho(\a_P) = e_k-e_l$.
\item $l > 0 > k$, and $\pi_\rho(\a_P) = -e_{-k}-e_{l}$, or 
\item  $0 > l >  k$, and $\pi_\rho(\a_P) = e_{-l} - e_{-k}$.
\end{enumerate} 
Therefore $\pi_\rho(\a_P)$ is not positive.

Next, we show that if $[-j,\,i] \not\in \Inv(\rho)$, then the image of $\a_P = e_i + e_j$ under $\pi_\rho$ is positive. If this were the case, then we would have one of the following:
\begin{enumerate}
\item $-k<l<0$, and $\pi_\rho(\a_P) = e_k - e_{-l}$,
\item $-k<0<l$, and $\pi_\rho(\a_P) = e_k + e_l$, or
\item $0 < -k < l$, and $\pi_\rho(\a_P) = e_l - e_{-k}$.
\end{enumerate}
Analagous to the previous case, if $[-j,\,i] \in \Inv(\rho)$ then the image of $\a_P$ is not positive. The cases to consider here are:
\begin{enumerate}
\item $l<-k<0$, and $\pi_\rho(\a_P) = e_k - e{-l}$,
\item $l<0<-k$, and $\pi_\rho(\a_P) = -e_k - e_l$, or
\item $0 < l < -k$, and $\pi_\rho(\a_P) = e_l - e_{-k}$.
\end{enumerate}

Finally, it is clear that $\pi_\rho(e_i) \notin \Phi^+$ if and only if $[i, \star] \in \Inv(\rho)$, as needed.\end{proof}

Recall that the weak left Bruhat order on $B_n$, with respect to the choice of positive roots $\Phi^+$, is the poset structure on $B_n$ with covering relations $w < w'$ for $w' = sw$ for some simple reflection $s$ with $l(w') = l(w) + 1$.

\begin{theorem}
$\phi$ defines a poset isomorphism $B_B(J_n,1) \to B_n$, where $B_n$ is ordered by the weak left Bruhat order. 
\end{theorem}
\begin{proof}  Notice that under the bijection $\phi$, the action of packet flips on $B_B(J_n, 1)$ corresponds to left multiplications by simple reflections.  The preceding lemma then shows that the covering relations in the two posets are identified under $\phi$, and the theorem follows.\end{proof}

\corollary{$\phi$ induces a bijection $A_B(J_n,2) \to R(w_0)$, where $w_0$ is the longest element of $B_n$ and $R(w_0)$ is the set of reduced expressions for $w_0$.  Under this bijection, two admissible orderings $\rho, \rho' \in A_B(J_n, 2)$ are elementarily equivalent if and only if the corresponding reduced expressions for $w_0$ are related by exchanging the order of a pair of adjacent commuting simple reflections.  Packet flip operations $p_K$ on $A_B(J_n, 2)$ are identified with $m = 3$ ($sts = tst$) braid relations for $K \in C_B^1(J_n, 3)$ and are identified with $m = 4$ ($stst = tsts$) braid relations for $K \in C_B^2(J_n, 3)$.}

\section{Appendix}

\subsection{Proof of Lemma \ref{l4}} \label{l4p}
As the lemma is merely casework, it was checked by a computer algorithm. We will describe this algorithm, and prove its correctness.

\defn{For some $\rho \in A_B(J_n,2)$, and two elements $a,\,b \in C_B(J_n,2)$, $a$ \emph{crosses $b$ in $\rho$} if there exists $\rho ' \in [\rho]$ such that the relative positions of $a$ and $b$ in the orders $\rho$ and $\rho'$ are opposite.

We first describe an algorithm which, on inputs $\rho \in A_B(J_n,2)$ and $a,\,b \in C_B(J_n,2)$, outputs 1 if $a$ crosses $b$ in $\rho$ and $0$ otherwise.

\alg If $b < a$ in $\rho$, replace $\rho$ by its reverse ordering.  Let $S$ denote the chain of elements in $\rho$ greater than $a$ and less than $b$. Initialize a list, called {\tt right}, containing only the element $a$. For each element $q$ in $S$, in ascending order, we compute whether $q$ commutes with every element in {\tt right}. If so, we continue. If not, we add $q$ to {\tt right}. Finally, return 1 if $b$ commutes with every $q$ contained in {\tt right}, and 0 otherwise.\endalg

\begin{proof}[Proof of correctness]  Suppose Algorithm 1 outputs 1.  Then each of the elements of $S \setminus\{a\}$  which is not added to ${\tt right}$ can be moved to the left past $a$, leaving only elements in ${\tt right}$ between $a$ and $b$.  But $b$ commutes with all elements in ${\tt right}$, so $b$ can be moved to the left past $a$, so $a$ crosses $b$ in $\rho$ as needed.

Conversely, suppose Algorithm 1 returns $0$.  Then there exists an element $q_1$ of ${\tt right}$ which does not commute with $b$.  Either $q_1$ does not commute with $a$, or there exists $q_2$ in right with $q_2 < q_1$ in the order $\rho$ such that $q_1$ and $q_2$ do not commute.  Continuing in this manner, there is a sequence $q_1, ..., q_s$ for some $s \geq 1$ of elements of ${\tt right}$ such that $a < q_s < \cdots q_1 < b$ in the order $\rho$ and each pair $(a, q_s)$, $(q_s, q_{s - 1})$, ..., $(q_2, q_1)$, $(q_1, b)$ does not commute.  It follows that the relative positions of these elements cannot change by transposing adjacent commuting elements, so in particular $a < b$ for all orders $\rho' \in [\rho]$, so $a$ does not cross $b$ in $\rho$.
\end{proof}

In the following algorithm, posets are represented as directed acyclic graphs, where vertices represent elements of the poset and there is a directed edge for every covering relation. For elements $a, b$, $a < b$ exactly when there is a directed path from $a$ to $b$. The \emph{transitive union} of two poset structures on the same set is given by the directed graph with the same vertex set and with edge set equal to the union of the edge sets for each poset.  The resulting relation is reflexive and transitive, and it is antisymmetric so long as it contains no cycles. Linear extensions are computed using the topological sorting algorithm.

\nota Recall that a $2$-packet $P$ is understood to be an ordered set with the ordering inherited from the standard ordering $\rho_{min}$. $\rev P$ is understood to be the same set, with the ordering relation inherited from $\rho_{max}$.

Recall that each case from Lemma \ref{l3} involves a sequence $Q$ of 4 or 5 elements of $C_B(J_n,2)$, all but one belonging to some packet $P_B(K)$ for $K \in C_B(J_n,\,3)$, and the remaining element shares exactly one index with $K$. As such, there is a unique element $R \in C_B(J_n, 4)$ such that the set $T = \cup_{S \in P_B(R)} P_B(S)$ contains the sequence $Q$. Furthermore, the unique 2-packet containing any pair of elements in the sequence $Q$ is contained in $T$.

\alg Initialize an empty list {\tt L}.  Each pair of elements ${\tt pair}$ appearing in the sequence $Q$ is contained in at most one common 2-packet $P$.  For each such ${\tt pair}$ appearing in $Q$ and lying in the 2-packet $P$:
\begin{itemize}
\item If {\tt pair} appears in $Q$ in standard order, add the poset given by the standard order on $P$ to {\tt L}.
\item Otherwise, add the poset $\rev P$ given by the reverse-standard order to {\tt L}.
\end{itemize}

Let $U$ be the set of packets $P_B(S)$ for $S \in P_B(R)$ whose order is not recorded in this manner. For each $B \subset U$, create a new list {\tt L'} containing the elements of {\tt L}. For each element $P \in U \cap B$, add the poset $P$ to {\tt L'}. For each element $P \in U \setminus B$, add the poset $\rev P$ to {\tt L'}. Compute the transitive  union over the relations in {\tt L'}. If there are no cycles, record a linear extension of the corresponding poset.

For each recorded linear extension, iterate over the packets $P_B(S)$ for $S \in P_B(R)$ until finding a 2-packet $P^*$ which is in standard order such that, with respect to the linear extension under consideration, either 
\begin{itemize}
\item $\min_\rho P^* > \min_\rho P_B(K)$ and $\min_\rho P_B(K)$ does not cross $\min_\rho P^* $, or
\item $\min_\rho P^* = \min_\rho P_B(K)$, $\max_\rho P^* < \max_\rho P_B(K)$, and $\max_\rho P^*$ does not cross $\max_\rho P_B(K)$.
\end{itemize} If this is the case, the algorithm continues. Otherwise, it outputs 0. If every linear extension recorded has been checked in this way, the algorithm outputs 1.

Algorithm 2 returns 1 when run on each of the cases in Lemma \ref{l3}, proving Lemma \ref{l4}.

\subsection{Proof of Theorem \ref{t2}} \label{t2p}
We first show that the image of the map \[ [\rho_{min}] = [\rho_1] \leq [\rho_2] \leq \cdots  [\rho_i] \cdots \leq [\rho_{max}] \mapsto K_1 < \cdots < K_m\t,\] in question lies in $A_B(J_n, k + 1)$.  For this, we need to check that for every element $K \in C_B(J_n, k + 2)$, its packet $P_B(K) \subset C_B(J_n, k + 1)$ appears in either standard or reverse standard order in $K_1 < \cdots < K_m$.  For this, we look at the restriction of the standard order $\rho_{min}$ to the set $S = \cup_{Z \in P_B(K)} P(Z) \subset C_B(J_n, k)$ and consider the possible orders in which the packets of elements of $P_B(K)$ could be flipped.  By inspection, we have the following tables indicating the possible orders in which the packets of elements of $P_B(K)$ can be flipped, which show in each case that the possible orders are exactly the standard or reverse standard order on $P_B(K)$.  A preferred representative is assumed in the left-hand column only.  For the case $K = [i,j, k]$ we consider only the subset $\{i, j, k\}$ of $S$, which is enough already to deduce the possible orderings of $P_B(K)$.

\begin{description}[leftmargin=*]
\item[Case $\bf k = 1$] \hfill \\
\smallskip

\noindent\begin{tabularx}{\textwidth}{|r|r|X|}

  \hline
  $K$ & Restriction of ${\rho_{min}}$ to ${S}$ & Possible flip sequence (up to reverse) \\[1ex] \hline \hline
  $[i,\,j,\,\star]$ & $-i < -j < j < i$ & $[i,\,j] \prec [i,\,\star] \prec [i,\,-j] \prec [j,\,\star]$ \\ \cline{1-3}
  $[i,\,j,\,k]$ with $j < 0$ & $i < j < k$ & $[i,\,j] \prec [i,\,k] \prec [j,\,k]$ \\ \cline{1-3}
  $[i,\,j,\,k]$ with $j > 0$ & $i < k < j$ & $[j,\,k] \prec [i,\,j] \prec [i,\,k]$ \\ \hline \end{tabularx}
  
\medskip

\item [Case $\bf k = 2$] \hfill \\

\vspace{-2ex}
\begin{figure}[H]
\makebox[\textwidth][c]{
\noindent{\small \begin{tabularx}{1.2\textwidth}{|p{2cm}|p{5cm}|X|}
  \hline
  $K$ & Restriction of ${\rho_{min}}$ to ${S}$ & Possible flip sequence (up to reverse) \\[1ex] \hline \hline
  $[i,\,j,\,k,\,\star]$ & $[-i,\,-j] < [-i,\,-k] < [-j,\,-k] < [i,\,\star] < [-i,\,j] < [-i,\,k] < [j,\,\star] < [-j,\,k] < [k,\,\star]$ & $[-i,\,-j,\,-k] \prec [i,\,j,\,\star] \prec [-i,\,-j,\,k] \prec [-i,\,-k,\,j] \prec[i,\,k,\,\star] \prec [j,\,k,\,\star] \prec [-i,\,j,\,k]$ \\ \cline{1-3}
  $[i,\,j,\,k,\,l]$, where $k < 0$ & $[i,\,j]<[i,\,k]<[j,\,k]<[i,\,l]<[j,\,l] < [k,\,l]$ & $[i\,j,\,k] \prec [i,\,j,\,l] \prec [i,\,k,\,l] \prec [j,\,k,\,l]$ \\ \cline{1-3}
  $[i,\,j,\,l,\,k]$, where $k > 0$ & $[i,\,j] < [-k,\,-l] < [i,\,l] < [i,\,k] < [j,\,l] < [j,\,k]$ & $[i,\,j,\,k] \prec [i,\,j,\,l] \prec [j,\,k,\,l] \prec [i,\,k,\,l]$\\
  \hline \end{tabularx}}}
  \end{figure}
\end{description}

So, we see that $K_1\cdots K_m$ is indeed an admissible ordering of $C_B(J_n, k + 1)$.

Next we show surjectivity.  Suppose $K_N\ldots K_1$ is an admissible order of $C_B(J_n, k + 1)$.  Let $r_0 = [\rho_{min}]$ denote the class of the standard ordering of $C_B(J_n, k)$.  We want to show that $K_N\ldots K_1$ gives a valid sequence of packet flips $p_{K_N}\ldots p_{K_1}$ on $r_0$.  With the empty sequence of packet flips as base case, assume inductively that $p_{K_i}\ldots p_{K_1}$ is a valid sequence of packet flips on $r_0$ for some $i \geq 0$.  Then writing $r_i = p_{K_i}\ldots p_{K_1}(r_0)$, we need to check that $K_{i + 1} \in N(r_i)$.  Noting that $K_{i + 1}$ is the minimal element of $C_B(J_n, k + 1)\setminus \Inv(r_i)$ with respect to the admissible order $K_N\ldots K_1$, it suffices to prove the following statement: If $\rho$ is an admissible ordering of $C_B(J_n, k)$ and $K \in C_B(J_n, k + 1) \setminus (\Inv(\rho) \cup N([\rho]))$, then $K$ is not minimal in the restriction of any admissible ordering to  $C_B(J_n, k + 1)\setminus \Inv(\rho)$. This is what we check by casework below.

\bigskip
\begin{description}[leftmargin=*]
\item[Case $\bf k = 1$] Let $\prec$ denote the ordering of $J_n = C_B(J_n, 1)$ given by $\rho \in A_B(J_n, 1)$.  Let $K = [k, l] \in C_B(J_n, 2)\setminus(\Inv(\rho) \cup N([\rho]))$ be as above, where if $K \in C_B^1(J_n, 2)$ then $[k, l]$ is a preferred representative, and if $K \in C_B^2(J_n, 2)$ then $l = -k > 0$, by the convention used previously.  As $[k, l] \notin \Inv(\rho) \cup N([\rho])$, there must exist $x \in J_n$ such that $k \prec x \prec l$.  The following table considers the possible relative orderings of $k, l,$ and $x$ under the usual ordering of $\bb Z$, denoted $<$.  The second column treats these relative positions of $x$, the third column lists implications about the order in which certain packets flips can be applied to $[\rho]$, and the final column lists the admissible order (up to reverse) of these packets, showing that each case leads to a contradication, as needed.

\setlength{\voffset}{-1in}

\smallskip

\begin{tabularx}{\textwidth}{|r|c|c|X|}
  \hline 
  Case & Condition & Implied order & Admissible order (up to reverse) \\
  \hline \hline
  $k \prec x \prec l$ & $x < k < l$ & $[x,\,k]\prec[k,\,l] \prec [x,\,l]$                 & $[x,\,k] < [x,\,l] < [k,\,l]$ \\ \cline{2-4}                   
                      & $k < x < l$ & $[k,\,l] \prec [k,\,x]$ and $[k,\,l] \prec [x,\,l]$ & $[k,\,x] < [k,\,l] < [x,\,l]$ \\ \cline{2-4}
                      & $k < l < x$ & $[l,\,x] \prec [k,\,l] \prec [k,\,x]$               & $[k,\,l] < [k,\,x] < [x,\,l]$ \\ \hline
\end{tabularx}
\vspace{3ex}
\bigskip
\item[Case $\bf k = 2$] We now make similar considerations for $k=2$. Here, when $K \in C_B^1(J_n, 3)$ we write $K = [k, l, m]$, and when $K \in C_B^2(J_n, 3)$ we write $K = [i, j, \star]$.  As in the case $k = 1$, each case leads to a contradiction, as needed, except here there are more cases to consider.  The cases are treated in the following table, and they form an exhaustive list of cases by Lemmas \ref{l2} and \ref{l3}.
\end{description}
\begin{landscape}
\small
\thispagestyle{empty}
\newgeometry{margin=1cm}
\setlength{\voffset}{3.132cm}

\begin{center}
\begin{figure}[htb]
\makebox[\textwidth][c]{

\begin{tabularx}{1.2868 \textwidth}{|c|p{2.5cm}|p{2cm}|l|c|} 
  \cline{1-5} 
  Case & condition & condition & Implied order of packet flips& Admissible order (up to reverse) \\ 
  \hline \hline $[k,\,l] \prec [k,\,x] \prec [k,\,m] \prec [l,\,m]$   & $k < l < m < 0$& $x > m$       & $[k,\,m,\,x] \prec [k,\,l,\,m] \prec [k,\,l,\,x]$ & $[k,\,l,\,m]                                                                                        < [k,\,l,\,x] < [k,\,m,\,x]$ \\ \cline{3-5}
                                                                        && $l < x < m$  & $[l,\,x,\,m] \prec [k,\,l,\,m] \prec[k,\,x,\,m]$ & $[k,\,l,\,m] < [k,\,x,\,m] < [l,\,x,\,m]$\\ \cline{3-5}
                                                                        && $k < x < l$  & $[k,\,l,\,m] \prec [k,\,x,\,m]$ and $[k,\,l,\,m] \prec [x,\,l,\,m]$ & $[k,\,x,\,m] < [k,\,l,\,m] < [x,\,l,\,m]$ \\ \cline{3-5}
                                                                        && $x < k$      & $[x,\,k,\,m] \prec [k,\,l,\,m] \prec [x,\,l,\,m]$ & $[x,\,k,\,m] < [x,\,l,\,m] < [k,\,l,\,m]$ \\ \cline{2-5}
                                                        & $k<l<0<m$     & $0<x<m$       & $[k,\,m,\,x] \prec [k,\,l,\,m] \prec [k,\,l,\,x]$ & $[k,\,l,\,m]                                     < [k,\,l,\,x] < [k,\,m,\,x]$ \\ \cline{3-5}
                                                                        && $x > m$      & $[k,\,l,\,m] \prec [k,\,l,\,x]$ and $[k,\,l,\,m] \prec [k,\,x,\,m]$ &  $[k,\,l,\,x] < [k,\,l,\,m] < [k,\,x,\,m]$ \\ \cline{3-5}
                                                                        && $l<x<0$      & same as $\uparrow$ & same as $\uparrow$ \\ \cline{3-5}
                                                                        && $k <x<l$     & $[k,\,x,\,l] \prec [k,\,l,\,m] \prec [k,\,x,\,m]$ & $[k,\,x,\,] < [k,\,x,\,m] < [k,\,l,\,m]$ \\ \cline{3-5}
                                                                        && $x < k$      & $[x,\,k,\,l] \prec [k,\,l,\,m] \prec [x,\,k,\,m]$ & $[x,\,k,\,l] < [x,\,k,\,m] < [k,\,l,\,m]$. \\ \cline{1-5}
  $[k,\,l] \prec [k,\,m] \prec [x,\,m] \prec [l,\,m]$   & $k<l<m<0$     & $x > m$       &$[l,\,m,\,x] \prec [k,\,l,\,m] \prec [k,\,m,\,x]$ & $[k,\,l,\,m] <                                                                                       [k,\,m,\,x] < [l,\,m,\,x]$ \\ \cline{3-5}
                                                                        && $l < x < m$  &$[l,\,x,\,m] \prec [k,\,l,\,m] \prec [k,\,x,\,m]$ & $[k,\,l,\,m] < [k,\,x,\,m] < [l,\,x,\,m]$ \\ \cline{3-5}
                                                                        && $k < x < l$  &$[k,\,l,\,m] \prec [k,\,x,\,m]$ and $[k,\,l,\,m] \prec [x,\,l,\,m]$ & $[k,\,x,\,m] < [k,\,l,\,m] < [x,\,l,\,m]$ \\ \cline{3-5}
                                                                        && $x < k$      &$[x,\,k,\,m] \prec [k,\,l,\,m] \prec [x,\,l,\,m]$ & $[x,\,k,\,m] < [x,\,l,\,m] < [k,\,l,\,m]$ \\ \cline{2-5}
                                                        & $k<l<0<m$     & $x > 0$       &$[l,\,m,\,x] \prec [k,\,l,\,m] \prec [k,\,m,\,x]$ & $[k,\,l,\,m] <                                   [k,\,m,\,x] < [l,\,m,\,x]$ \\ \cline{3-5}
                                                                        && $l < x < 0$  &$[l,\,x,\,m] \prec [k,\,l,\,m] \prec [k,\,x,\,m]$ & $[k,\,l,\,m] < [k,\,x,\,m] < [l,\,x,\,m]$ \\ \cline{3-5}
                                                                        && $k < x < l$  &$[k,\,l,\,m] \prec [k,\,x,\,m]$ and $[k,\,l,\,m] \prec [j,\,x,\,m]$ & $[k,\,x,\,m] < [k,\,l,\,m] < [j,\,x,\,m]$ \\ \cline{3-5}
                                                                        && $x < k$      &$[x,\,k,\,m] \prec [k,\,l,\,m] \prec [x,\,l,\,m]$ & $[x,\,k,\,m] < [x,\,l,\,m] < [k,\,l,\,m]$ \\ \cline{1-5}
  $[k,\,l] \prec [l,\,x] \prec [l,\,m]$                 & $k<m<0$       &$x > m$        & $[l,\,m,\,x] \prec [k,\,l,\,m] \prec [k,\,l,\,x]$ & $[k,\                                                                                         ,l,\,m] < [k,\,l,\,x] < [l,\,m,\,x]$ \\ \cline{3-5}
                                                                        &&$k < x < m$   & $[k,\,l,\,m] \prec [k,\,l,\,x]$ and $[k,\,l,\,m] \prec [l,\,x,\,m]$ & $[k,\,l,\,x] < [k,\,l,\,m] < [l,\,x,\,m]$ \\ \cline{3-5}
                                                                        &&$x < k$       & $[x,\,k,\,m] \prec [k,\,l,\,m] \prec [x,\,l,\,m]$ & $[x,\,k,\,m] < [x,\,l,\,m] < [k,\,l,\,m]$. \\ \cline{2-5}
                                                        & $k<0<m $      &$x > m$        & $[k,\,l,\,m] \prec [k,\,l,\,x]$ and $[k,\,l,\,m] \prec [l,\,m,\,x]$ & $[k,\,l,\,x]<[k,\,l,\,m]<[l,\,m,\,x]$ \\ \cline{3-5}
                                                                        && $k < x < 0$  & same as $\uparrow$ & same as $\uparrow$ \\ \cline{3-5}
                                                                        && $0 < x < m$  & $[l,\,m,\,x] \prec [k,\,l,\,m] \prec [k,\,l,\,x]$ & $[k,\,l,\,m] < [k,\,l,\,x] < [l,\,m,\,x]$\\ \cline{3-5}
                                                                        && $x < k$      & $[x,\,k,\,l] \prec [k,\,l,\,m] \prec [x,\,l,\,m]$ & $[x,\,k,\,l]<[x,\,l,\,m]<[k,\,l,\,m]$ \\ \cline{1-5}
 $[i,\,j] \prec [i,\,x] \prec [i,\,\star] \prec [i,\,-j] \prec [j,\,\star]$&$i<j<0$& $x < i$& $[x,\,i,\,j] \prec [-i,\,-j,\,\star] \prec [-x,\,-i,\,\star]$& $[x,\,i,\,j] < [-x,\,-i,\,\star]<[-i,\,-j,\,\star]$ \\ \cline{3-5}
                                                                          &&$i<x<j$& $[i,\,x,\,j]\prec[-i,\,-j,\,\star] \prec [-i,\,-x,\,\star]$&$[i,\,x,\,j]<[-i,\,-x,\,\star]<[-i,\,-j,\,\star]$ \\ \cline{3-5}
                                                                          &&$j<x<0$&$[-i,\,-j,\,\star] \prec [i,\,j,\,x]$ and $[-i,\,-j,\,\star] \prec [i,\,x,\,-j]$&$[i,\,j,\,x]<[-i,\,-j,\,\star]<[i,\,x,\,-j]$\\ \cline{3-5}
                                                                          &&$0<x<-i$&$[-i,\,x,\,\star] \prec [-i,\,-j,\,\star] \prec [i,\,j,\,x]$&$[-i,\,-j,\,\star] < [i,\,j,\,x] < [-i,\,x,\,\star]$ \\ \cline{3-5}
                                                                          &&$x>-i$&$[x,\,-i,\,\star] \prec [-i,\,-j,\,\star] \prec [x,\,-i,\,j]$&$[x,\,-i,\,\star] < [x,\,-i,\,j] < [-i,\,-j,\,\star]$ \\ \cline{1-5}
$[i,\,j] \prec [i,\,\star] \prec [i,\,x] \prec [i,\,-j] \prec [j,\,\star]$&$i<j<0$& $x < j$&$[-x,\,-i,\,\star] \prec [-i,\,-j,\,\star] \prec [x,\,i,\,-j]$&$[-x,\,-i,\,\star]<[x,\,i,\,-j]<[-i,\,-j,\,\star]$  \\ \cline{3-5}
                                                                          &&$j<x<0$& $[-i,\,-j,\,\star] \prec [i,\,j,\,x]$ and $[-i,\,-j,\,\star] \prec [i,\,x,\,-j]$&$[i,\,j,\,x] < [-i,\,-j,\,\star] < [i,\,x,\,-j]$ \\ \cline{3-5}
                                                                          &&$0<x<-j$& $[i,\,-j,\,x] \prec [-i,\,-j,\,\star] \prec [i,\,j,\,x]$&$[-i,\,-j,\,\star] < [i,\,j,\,x] < [i,\,-j,\,x]$ \\ \cline{3-5}
                                                                          &&$-j<x<-i$& $[-i,\,-j,\,\star] \prec [i,\,x,\,-j]$ and $[-i,\,-j,\,\star] \prec [-i,\,x,\,\star]$ & $[-i,\,x,\,\star]<[-i,\,-j,\,\star]<[i,\,x,\,-j]$  \\ \cline{3-5}
                                                                          &&$x>-i$&$[-i,\,-j,\,\star] \prec [-x,\,i,\,-j]$ and $[-i,\,-j,\,\star] \prec [-x,\,-i,\,-j]$&$[-x,\,i,\,-j]<[-i,\,-j,\,\star] < [-x,\,-i,\,-j]$ \\ \cline{1-5}
$[i,\,j] \prec [i,\,\star] \prec [j,\,x] \prec [i,\,-j] \prec [j,\,\star]$&$i<j<0$& $x < i$&$[x,\,i,\,j] \prec [-i,\,-j,\,\star] \prec [x,\,j,\,-i]$&$[x,\,i,\,j]<[x,\,j,\,-i]<[-i,\,-j,\,\star]$  \\ \cline{3-5}
                                                                          &&$i<x<j$& $[-i,\,-j,\,\star] \prec [i,\,j,\,x]$ and $[-i,\,-j,\,\star] \prec [-x,\,-j,\,\star]$& $[i,\,j,\,x]<[-i,\,-j,\,\star]<[-x,\,-j,\,\star]$  \\ \cline{3-5}
                                                                          &&$j<x<0$&$[-i,\,-j,\,\star] \prec [i,\,j,\,x]$ and $[-i,\,-j,\,\star] \prec [i,\,-j,\,-x]$&$[i,\,j,\,x]<[-i,\,-j,\,\star]<[i,\,-j,\,-x]$  \\ \cline{3-5}
                                                                          &&$0<x<-j$&$[-j,\,x,\,\star]\prec[-i,\,-j,\,\star] \prec [i,\,j,\,x]$ & $[-i,\,-j,\,\star]<[i,\,j,\,x]<[-j,\,x,\,\star]$ \\ \cline{3-5}
                                                                          &&$-j<x<-i$&$[i,\,x,\,-j] \prec [-i,\,-j,\,\star] \prec [i,\,j,\,x]$&$[i,\,x,\,-j]<[i,\,j,\,x]<[-i,\,-j,\,\star]$\\ \cline{3-5}
                                                                          &&$x>-i$&$[-i,\,-j,\,\star] \prec [-x,\,-i,\,-j]$ and $[-i,\,-j,\,\star] \prec [-x,\,i,\,-j]$&$[-x,\,i,\,-j]<[-i,\,-j,\,\star]<[-x,\,-i,\,-j]$ \\ \cline{1-5}
$[i,\,j] \prec [i,\,-j] \prec [j,\,x] \prec [j,\,\star]$&$i<j<0$& $x < i$& $[x,\,j,\,i] \prec [-i,\,-j,\,\star] \prec [x,\,j,\,\star]$&$[x,\,j,\,i] < [x,\,j,\,\star] < [-i,\,-j,\,\star]$  \\ \cline{3-5}
                                                                          &&$i<x<j$&$[-i,\,-j,\,\star] \prec [i,\,x,\,j]$ and $[-i,\,-j,\,\star] \prec [x,\,j,\,-i]$&$[i,\,x,\,j]<[-i,\,-j,\,\star] < [x,\,j,\,-i]$  \\ \cline{3-5}
                                                                          &&$j<x<0$&$[-i,-j,\,\star] \prec [i,\,j,\,x]$ and $[-i,-j,\,\star] \prec [i,\,-j,\,-x]$&$ [i,\,j,\,x]< [-i,-j,\,\star] < [i,\,-j,\,-x]$  \\ \cline{3-5}
                                                                          &&$0<x<-j$&$[i,\,-x,\,-j] \prec [-i,\,-j,\,\star] \prec [-j,\,x,\,\star]$&$[-i,\,-j,\,\star]<[i,\,-x,\,-j] < [-j,\,x,\,\star]$ \\ \cline{3-5}
                                                                          &&$-j<x<-i$& $[i,\,-x,\,-j]\prec[-i,\,-j,\,\star]\prec[i,\,j,\,x]$& $[i,\,-x,\,-j]<[i,\,j,\,x]<[-i,\,-j,\,\star]$ \\ \cline{3-5}  
                                                                          &&$x > -i$& $[-x,\,i,\,-j] \prec [-i,\,-j,\,\star] \prec [x,\,-j,\,\star]$&$[-x,\,i,\,-j]<[x,\,-j,\,\star]<[-i,\,-j,\,\star]$ \\ \cline{1-5}                                                                                                                                                                                                                                                       
\end{tabularx}}

\end{figure}
\end{center}
\end{landscape}
\normalsize

\newpage

\section{Acknowledgements}

This paper represents the results of an undergraduate research project conducted by Suhas Vijaykumar in MIT's Summer Program in Undergraduate Research, with graduate student mentor Seth Shelley-Abrahamson.  We thank Ben Elias for suggesting the idea of generalizing the Manin-Schechtman higher Bruhat orders to type B and for his guidance and excitement throughout the completion of this work, and we thank Daniel Thompson and Gabriella Studt for their initial work on this project.

\end{document}